\documentclass[12pt]{article}
\usepackage{amsmath}
\usepackage{amsfonts}
\usepackage{amsthm}
\usepackage{authblk}
\usepackage{algorithm}
\usepackage{algpseudocode}
\usepackage{hyperref}
\usepackage{verbatim}
\usepackage{float}

\theoremstyle{plain}
\newtheorem{theorem}{Theorem}
\newtheorem{lemma}[theorem]{Lemma}
\theoremstyle{definition}
\newtheorem{definition}{Definition}
\theoremstyle{remark}

\newcommand{\ZZ}{\mathbb{Z}}
\newcommand{\NN}{\mathbb{N}}
\newcommand{\seqvar}{\sigma}
\newcommand{\differenceset}[1]{D_{#1}}
\newcommand{\partialsum}[2]{s(#1,#2)}
\newcommand{\partialsumset}[1]{\Sigma_{#1}}
\newcommand{\seqnum}[1]{\href{https://oeis.org/#1}{\textrm \underline{#1}}}

\title{A More Efficient Algorithm for Finding the Number of Permutations of $\ZZ/n\ZZ$ with Distinct Partial Sums}
\author{Baker, Quinn\\
\href{mailto:quinn.baker@ucalgary.ca}{quinn.baker@ucalgary.ca}\\
Department of Computer Science,\\
University of Calgary\\
Calgary, Alberta, Canada, T2N 1N4
\and 
Feaver, Amy \\
\href{mailto:feaveral@grace.edu}{feaveral@grace.edu}\\
Department of Computer Science,\\
Grace College\\
Winona Lake, Indiana, 46590
}

\date{}    
\begin{document}

\begin{titlepage}
\maketitle
\end{titlepage}

\begin{abstract}The notion of the factorial extends to abelian groups under the group operation and any permutation of the elements of the group. Notably, for some finite groups, it is possible to order elements in these groups such that no two ``factorials'' are the same. In this paper we count the number of ways to order elements to achieve this result for the integers modulo 20 and 22 under addition,  as well as introduce an improved algorithm to count these permutations for the integers modulo $n$.  We relate these numbers to the values in an already known sequence via a bijection which we prove in this paper. This proof and other lemmas allow us to calculate new terms for the sequence.
\end{abstract}

\section{Introduction}

In this paper, we study the sequence numbered \seqnum{A141599} in the On-Line Encyclopedia of Integer Sequences (OEIS)\cite{OEIS}. This sequence arises from counting permutations of $\ZZ/n\ZZ$, $n$ even, with specific properties defined in subsection \ref{traddef}.  In particular, we find the 10th, 11th and 12th terms as: 5,074,931,072, 298,557,044,000 and 21,224,961,132,544  respectively using the algorithm described in Section \ref{sec:algo}. Recently, the 10th and 11th terms were verified independently on the OEIS by Bert Dobbelaere.

We show the equivalence between the sequence's original definition and a modern, more generalized definition in order to unify classical results with our current approach to the problem. 

\subsection{Sequence definition}\label{traddef}

The sequence \seqnum{A141599} is defined to be ``the number of difference sets for permutations of $[2k]$ with distinct differences." Here, $[2k]$ is referring to $\ZZ/2k\ZZ$. 

For a fixed $n\in\ZZ^+$ let $\seqvar$ be a permutation of $\ZZ/n\ZZ$,
\[\seqvar := (\seqvar_0, \seqvar_1,\ldots,\seqvar_{n-1}).\]

The \textit{differences} of $\seqvar$ are the $n-1$ values $d_i$ given by
\[d_i\equiv\seqvar_i-\seqvar_{i-1}\pmod{n},\] 
for $1\leq i\leq n-1$. The \textit{difference set} of $\sigma$ is $\differenceset{\seqvar}:=(d_1, d_2,\ldots,d_{n-1})$, and $\sigma$ is said to have \textit{distinct differences} if each $d_i$ is unique, that is, if $\{d_i:1\leq i\leq n-1\}=\ZZ/n\ZZ\setminus\{0\}$. Note that the term ``difference set'' is a misnomer, as the elements of a difference set are ordered. However, we use this terminology as it was introduced in Gilbert\cite{Gil}. 

Distinct differences can only exist if $n$ is even \cite{Baker}. As an example, the second term of the above sequence is 2. This is because when $k=2$ there are exactly two difference sets of $\ZZ/4\ZZ$ with distinct differences for permutations of $\ZZ / 4\ZZ$. These difference sets are $(1, 2, 3)$ and $(3, 2, 1)$.  For example, two of the permutations that produce differences $(1, 2, 3)$ are $\seqvar = (0,1,3,2)$ and $\seqvar'=(2,3,1,0)$, confirmed by taking the differences of each adjacent pair of elements, and we have equivalent difference sets
\[D_{\seqvar}=D_{\seqvar'}=(1,2,3).\]
An exhaustive search of all permutations of $\ZZ / 4\ZZ$ would confirm that $(1,2,3)$ and $(3,2,1)$ are the only inequivalent difference sets with distinct differences. We can see that the second term of the sequence is 2 as a result.

\subsection{Sequencings of \texorpdfstring{$\ZZ/n\ZZ$}{Z/nZ}}

We now define sequenceable groups, as introduced by Gordon\cite{Gor}.

\begin{definition}[Sequenceability] A finite group $G$ of order $n$ with binary operation $\cdot$ is  \textit{sequenceable} if the elements of $G$ can be arranged in an ordered sequence $g_1,g_2,\ldots,g_n$ such that the partial ordered products $g_1\cdot g_2\cdot g_3\cdots g_i$ are all distinct for $1\leq i\leq n$.
\end{definition}

\begin{definition}[Sequencing]
A \textit{sequencing} of $G$ is a permutation of its elements which shows that $G$ is sequenceable. That is, the partial products of the elements of $G$ under this permutation are all distinct. 
\end{definition}

We introduce new notation for ease of reference. Let $\sigma$ be a permutation of $\ZZ/n\ZZ$ as before.
\begin{definition}
The $i$th \textit{partial product} of $\ZZ/n\ZZ$ under the permutation $\sigma$ is $$\partialsum{\sigma}{i} := g_1\cdot g_2\cdot g_3\cdots\cdot g_i$$
\end{definition}

\begin{definition}
The set of partial products of $G$ under $\sigma$ is $$\partialsumset{\sigma}:=\{\partialsum{\sigma}{i}\ |\ 1\leq i \leq n\}.$$ 
\end{definition}

 From these definitions, it follows that a finite group $G$ is sequenceable if and only if there exists a permutation $\sigma$ such that $\partialsumset{\sigma} = G$.

For even $n$, we claim that the number of sequencings of $\ZZ/n\ZZ$ is equal to the number of inequivalent difference sets for $\ZZ/n\ZZ$, and thus we can use properties of sequencings to assist in computing new terms of OEIS sequence \seqnum{A141599}.

As an example, we showed that there are only two difference sets corresponding to $n=4$ that have distinct differences,  $(1,2,3)$ and $(3, 2, 1)$. Thus there should be exactly two sequencings of $\ZZ / 4\ZZ$. Through an exhaustive search, one would find that these are $t_1 = (0, 1, 2, 3)$ and $t_2 = (0, 3, 2, 1)$. 

\subsection{Contributions}

In this paper, we show that for a given $n$ there is a bijection between the number of difference sets of size $n$ with distinct differences and the number of sequencings mod $n$ for any $n\in\mathbb{N}$. We utilize this in order to find new terms of the sequence \seqnum{A141599}. We prove new properties of sequencings and use them in our recursive algorithm which finds terms of \seqnum{A141599}. This is the most efficient algorithm known to the authors to date.

\section{Properties of sequencings}

\begin{theorem}
There is a bijection between sequencings of $\ZZ/n\ZZ$ and difference sets for permutations of $\ZZ/n\ZZ$ with distinct differences.
\end{theorem}

\begin{proof}
The first half of this proof will show we can construct a sequencing of $\ZZ/n\ZZ$ from a difference set with distinct differences, and the second half will show we can construct a permutation whose difference set has distinct differences from a sequencing of $\ZZ/n\ZZ$.
\par
First, let $\differenceset{\tau}=(d_1, d_2,\ldots d_{n-1})$ be a difference set of a permutation $\tau$ of $\ZZ/n\ZZ$ with distinct differences. We use a permutation corresponding to this difference set to show that the new permutation given by $\sigma := (0, d_1, d_2,\ldots d_{n-1})$ is a sequencing for $\ZZ/n\ZZ$.

This difference set corresponds to the permutation $\tau$, which can be constructed as follows:
$\tau_0 = 0,$
$\tau_1 = d_1,$
$\tau_i = d_1+d_2+\cdots + d_i,$ for  $2\leq i \leq n-1.$

Now consider $\sigma = (0, d_1, d_2,\ldots, d_{n-1})$. The set of partial sums of this sequence is 
\begin{align*}
\partialsumset{\sigma} &= \{0, d_1, d_1+d_2,\ldots,d_1+\cdots+ d_{n-1}\} \\
&= \{\tau_0,\tau_1,\ldots,\tau_{n-1}\}\\
&= \ZZ/n\ZZ,
\end{align*}
since $\tau$ is a permutation of $\ZZ/n\ZZ$. Therefore, $(0, d_1, d_2,\ldots, d_{n-1})$ is a sequencing of $\ZZ/n\ZZ$.

\par
Second, let $\sigma'=(0,g_1,g_2, \ldots, g_{n-1})$ be a sequencing of $\ZZ/n\ZZ$. From this, we can calculate its set of partial sums, $\partialsumset{\sigma'} = \{0, g_1, g_1+g_2, \ldots, g_1+g_2+\cdots+g_{n-1}\}$. Since $\sigma'$ is a sequencing, the partial sum set $\partialsumset{\sigma'}$ contains all elements of $\ZZ/n\ZZ$. Let $\tau' = (0, g_1, g_1+g_2, \ldots, g_1+g_2+g_{n-1})$ Calculating its difference set $D_{\tau'}$, we get $\{g_1, g_2, g_3, \ldots, g_{n-1}\}$. Since these are the elements of the sequencing $\sigma'$, we know they are unique, and so this is a difference set with distinct differences.

\end{proof}

Now that we have established a one-to-one correspondence between difference sets with distinct differences and sequencings, we prove more properties of sequencings with the goal of making our algorithm in Section \ref{sec:algo} more efficient.

\begin{lemma}\label{lem:rel_prime_constructive} Let $\sigma$ be any permutation of $\ZZ / n\ZZ$ and let $r$ be any element of $\ZZ / n\ZZ$ such that $\gcd(r,n)=1$. Then $(\sigma_0,\sigma_1,\ldots,\sigma_{n-1})$ is a sequencing if and only if $(r\sigma_0, r\sigma_1,\ldots,r\sigma_{n-1})$ is.
\end{lemma}

\begin{proof}
For $n\in\NN$ assume that for some permutation $\sigma$, the ordering $(\sigma_i)_{i=0}^{n-1}$ is a sequencing. Then the set of partial sums $\Sigma_\sigma$
is equal to $\ZZ / n\ZZ$. 

Since $r$ is relatively prime to $n$, then multiplying all of the elements of $\ZZ / n\ZZ$ by $r$ gives a permutation of $\ZZ / n\ZZ$. Thus 

\begin{align*} \ZZ / n\ZZ &= \bigl\{r\cdot s(\sigma,i)\pmod{n}\mid 0\leq i \leq n-1\bigr\} \\ &=\Biggl\{\sum_{j=0}^i (r\cdot\sigma_j)\pmod{n}\mid 0\leq i \leq n-1\Biggr\} 
\end{align*}
so multiplying the ordering given by $\sigma$, element-by-element with $r$ also gives a permutation and thus a sequencing.

Further, if $r$ is relatively prime to $n$, it has a multiplicative inverse $r^{-1}\bmod{n}$ which is also relatively prime to $n$. Thus the ordering given by 
\[\bigl(r^{-1}r\sigma_i\bigr)_{i=0}^{n-1}=\bigl(\sigma_i\bigr)_{i=0}^{n-1}\]
is also a sequencing since we are multiplying a known sequencing by a relatively prime element.

\end{proof}

The next lemma is used to significantly reduce the number of computations needed to count sequencings. We define \textit{branches} which will show up later in our algorithm. Let $P$ be the set of all possible permutations of $\ZZ/n\ZZ$.

\begin{definition}\label{branches}
For each $i\in(\ZZ/n\ZZ)^+$ we define the $i$th \emph{branch} of $P$ as 
\[P_i:= \{(0,i,g_2,\ldots,g_{n-1})\in P \}.\]  
That is, the set $P_i$ is the set of all permutations whose first non-zero element is $i$.
\end{definition}

Thus,  $P=\cup_{i=1}^{n-1}P_i$. The motivation for the lemma below is that if we consider the branch $P_1$, we can find $P_a$, if $\gcd(a,n)=1$, not by constructing $P_a$, but by taking $P_1$ and multiplying each element of the permutation by $a\bmod{n}$.

Similarly, we can find all sequencings in this manner, rather than computing them directly. We will see in Lemma \ref{lem:branch} that all sequencings in $P_a$ with $\gcd(a,n)=1$ can be found by multiplying all sequencings beginning with $(0,1,\ldots)$ by $a\bmod{n}$. This technique is used to significantly cut down computation time by simply multiplying the number of elements of $P_1$ by $\varphi(n)$.

\begin{lemma}\label{lem:branch}
Let $C$ be the set of all sequencings of $\ZZ / n\ZZ$. Partition $C$ into branches of sequencings $C_i$ with $1\leq i\leq n-1$, where
\[C_i:= \{(0,i,g_2,\ldots,g_{n-1})\in C \}\]
and let $D$ be the set of all proper divisors of $n$. Then \[\#C = \sum_{d\in D}\#C_d\cdot \varphi(n/d).\]
\end{lemma}

\begin{proof}

For each $d\in D$ define

\[
M_d = \{1 \leq i \leq n/d \mid (i,n)=1\}.
\]

It is a well-established fact that $\varphi(n/d) = \#\{1 \leq a \leq n \mid (a,n)=d\}$, and thus we find that
\begin{align*}
\varphi(n/d) &= \lvert\{1 \leq a \leq n \mid (a,n)=d\}\rvert \\
&=  \lvert  \{ 1 \leq a/d \leq n/d \mid (a/d,n/d)=1\}\rvert \\
&= \lvert \{1 \leq i \leq n/d \mid (i,n)=1\}\rvert \\
&=\# M_d.
\end{align*}

Let \[W:=\bigcup_{d\in D} \bigcup_{g\in M_d} (C_d \cdot g).\]

Then $\#W=\sum_{d\in D}\#C_d\cdot \varphi(n/d)$. Thus we will prove this lemma by proving $W=C$.

We first establish that $W\subseteq C$. This follows immediately from Lemma \ref{lem:rel_prime_constructive} which demonstrates that multiplying a sequencing by an element relatively prime to $n$ still gives a sequencing. 

We next show that $C\subseteq W$. Pick any $t\in C$. Then $t$ is a sequencing $c\in C$ and write $c=\{0,g_1,g_2,\ldots,g_{n-1}\}.$ The first nonzero element of $c$, $g_1$ can be factored into $g_1\equiv g\cdot d$ (mod $n$) for some $d\in D$. 

As $g$ is relatively prime to $n$, its inverse $g^{-1}\in\ZZ/n\ZZ$ exists, is relatively prime to $n$ and $g_1g^{-1}\equiv d\ $ (mod $n$). Now consider
$\{0,g_1g^{-1},g_2g^{-1},\ldots,g_{n-1}g^{-1}$.
This  is a sequencing and it starts with $d$ as the first nonzero element so $\{0,g_1g^{-1},g_2g^{-1},\ldots,g_{n-1}g^{-1}\}\in C_d$. Therefore, \[c=\{0,g_1g^{-1},g_2g^{-1},\ldots,g_{n-1}g^{-1}\}\cdot g.\] As this is the form of all elements of $W$ we have that $c\in W$ and thus $C\subseteq W$.

Therefore we have $C=W$.

\end{proof}

\section{Algorithm}\label{sec:algo} 
We have improved the brute-force sequencing-finder algorithm from Baker\cite{Baker}. To explain the improvements, we first introduce the basic algorithm used to determine if a particular permutation is a sequencing. 
\begin{algorithm}[H]
\caption{Sequencing Validator}\label{alg:seqvalidator}
\begin{algorithmic}[1]
\Procedure{sequencingValidator}{$n,permutation$}
\State $sum := 0$
\State $seenVals := $ empty list
\For{$val$ in $permutation$}
\State $sum := (sum + val) \bmod{n}$
\If{$sum$ is already in $seenVals$}
\State \textbf{return} false
\EndIf
\State Add $sum$ to $seenVals$
\EndFor
\State \textbf{return} true
\EndProcedure
\end{algorithmic}
\end{algorithm}

As well, we define a \textit{partial permutation}.
\begin{definition}[Partial Permutation]\label{def:partialperm}
    A partial permutation of $\ZZ/n\ZZ$ is an ordered list of elements $(g_1, g_2, \ldots, g_i)$, $i<n$.
\end{definition}
As the algorithm from \cite{Baker} was a purely brute force algorithm, no information was retained from validating one permutation to another. However, we can in fact use information about partial permutations that are known not to be valid. Since each partial product must be distinct in order for a permutation to be a sequencing, once two partial products are found to be the same, we can rule out that partial permutation as the beginning of a sequencing, as shown in Algorithm \ref{alg:seqvalidator}. This means that every permutation beginning with this partial permutation cannot possibly be a sequencing of $\ZZ/n\ZZ$. Our algorithm improves on previous brute-force algorithms by being able to ``cut off" this family of known bad permutations, preventing their full evaluation using Algorithm \ref{alg:seqvalidator} 1, saving many wasted calculations. We have written a recursive algorithm which constructs and examines each partial permutation as elements are added. If a partial permutation is found to be invalid, then that recursive branch is abandoned. As an example, for $\ZZ / 20\ZZ$, we know the partial permutation $\sigma=(0, 1, 19)$ is already invalid, as $\partialsum{\sigma}{2} = \partialsum{\sigma}{0} = 0$. Our algorithm will thus never check any partial permutations beginning with $(0, 1, 19)$, saving $17!$ calls to Algorithm \ref{alg:seqvalidator}. In addition, our algorithm introduces no new significant overhead, thus acting as a strict improvement over the brute-force algorithm in terms of counting performance. The only downside from prior brute-force algorithms is the loss of knowledge about the sequencings themselves, now only counting their number.

\begin{algorithm}[H]
\caption{Recursive Branch Search}\label{alg:branchrecurse}
\begin{algorithmic}[1]
\Procedure{branchSearch}{$n,sum,partials,rem,total$}\Comment{see Def.\ \ref{var_names}}
\If{$sum\equiv0$ or ($sum\equiv \frac{n}{2}$ and $length(rem) > 1$)} 
\State \textbf{return}
\EndIf
\If{$length(rem)$ is 0}
\State total += 1
\State \textbf{return}
\EndIf
\For{$val$ in $rem$}
\State Remove $val$ from $rem$
\State $sum :\equiv (sum + val) \bmod{n}$
\If{$partials[sum]$ = False} \Comment{Only recurse if this partial ordering is valid}
\State $partials[sum] :=$ True
\State $\textsc{branchSearch}(n, sum, partials, rem, total)$
\EndIf
\State Reset partials, sum, and rem to their values at the start of the loop.
\EndFor
\EndProcedure
\end{algorithmic}
\end{algorithm}

\par This algorithm represents one thread of the recursive algorithm, which processes $P_i$,  such that $i$ is a divisor of $n$ less than $\frac{n}{2}$. 

\begin{definition}\label{var_names}
The parameters to this function represent a partial ordering and all possible ``next values'' in that partial ordering. Below are the definitions for each parameter.
\begin{itemize}
    \item \textit{n}: the modulus, the size of the group the ordering is on. This value never changes between recursive calls. 
    \item \textit{sum}: the sum of all elements of the current partial ordering.
    \item \textit{partials}: an array that notes the partial sums in  this partial ordering. This can be achieved either by a traditional array of values, or as a boolean lookup table, with each index representing whether a partial sum of that index is present. The latter is used in our implementation of the algorithm.
    \item \textit{rem}: an array of remaining values of $\ZZ/n\ZZ$ not yet used in the partial order this call represents. In the initial call, this array is filled with all values of $\ZZ/n\ZZ$ with the exception of $0$ and $d$, where the branch starts with $0,d$ as the first two values in the sequencing. With each recursive call, \textit{rem} is made smaller as the partial ordering grows.
    Each value of this array will be used to recursively call this function with new states representing the partial permutation with each element added to the current state.

    \item \textit{total}: the total number of sequencings found from the initial call. This value is shared across all recursive calls.

\end{itemize}
\end{definition}

As an example, if we were to use this algorithm to find the 3rd term of the sequence A141599, the parameter $n$ would be 6. We must select a a branch to start with; this always involves beginning with 0 and then selecting a divisor of $6$ less than $3$. For the sake of this example we will consider the branch that begins with the sequence \{0,2\}. Thus $\texttt{sum} \equiv 0+2 \pmod6 = 2$. \texttt{partials} indicates the partial sums that have occurred so far, i.e., 0 and 2. The array is \texttt{True} in the 0th and 2nd position, \texttt{False} otherwise. The variable \texttt{rem} will contain all values that do not already appear in the branch, which are 1, 3, 4, and 5. Finally, $total$ always starts at 0 and increments as sequencings are found.

Thus the initial call to this function in the scenario described above to 
\textsc{branchSearch} would have the following parameters:
\begin{itemize}
\item \texttt{n}: $6$
\item \texttt{sum}: $2$
\item \texttt{partials}: \{True, False, True, False, False, False\}
\item \texttt{rem}: $\{1,3,4,5\}$
\item \texttt{total}: $0$
\end{itemize}

An exact analysis of this algorithm's runtime is unknown, however in practice it has performed significantly faster than a simple brute force search.  Reviewing known results we can infer certain properties about the ideal algorithm for finding sequencings.

\begin{table}[H]
    \begin{center}
    \begin{tabular}{|c|c|c|}
            \hline
            $n$ & Sequencings of $\ZZ/n\ZZ$ & Permutations of $\ZZ/n\ZZ$\\
            \hline
            $2$ & $1$ & $2$ \\
            $4$ & $2$ & $24$ \\
            $6$ & $4$ & $720$ \\
            $8$ & $24$ & $40320$ \\
            $10$ & $288$ & $3628800$ \\
            $12$ & $3856$ & $479001600$ \\
            $14$ & $89328$ & $87178291200$ \\
            $16$ & $2755968$ & $20922789888000$ \\
            $18$  & $103653120$ & $6402373705728000$ \\
            $20$ & $5074931072$ & $2432902008176640000$ \\
            $22$ & $298557044000$ & $1124000727777607680000$ \\
            \hline
        \end{tabular}
    \end{center}
    \caption{Table showing the counts of sequencings and total number of permutations of $\ZZ/n\ZZ$ for values of $n$ where the number of sequencings is known and non-zero.}
    \label{tab:counts}
\end{table}

\begin{table}[H]
    \begin{center}
        \begin{tabular}{|c|c|}
        \hline
         $n$ &  Ratio of columns 2 and 3 in Table \ref{tab:counts}\\
         \hline
         $2$ & $0.5$\\
         $4$ & $0.083$\\
         $6$ & $0.0056$\\
         $8$ & $5.95\cdot10^{-4}$\\
         $10$ & $7.94\cdot 10^{-5}$\\
         $12$ & $8.05\cdot 10^{-6}$\\
         $14$ & $1.02\cdot10^{-6}$ \\
         $16$ & $1.32\cdot10^{-7}$\\
         $18$ & $1.62\cdot10^{-8}$\\
         $20$ & $2.09\cdot10^{-9}$\\
         $22$ & $2.66\cdot10^{-10}$\\
         \hline
    \end{tabular}
    \end{center}
    \caption{Table showing the ratio of sequencings to permutations of $\ZZ/n\ZZ$ for values of $n$ where the number of sequencings is known and non-zero. The ratios are approximate.}
    \label{tab:ratio}
\end{table}

The above tables reveal that the ratio of sequencings in the space of all permutations of $\ZZ/n\ZZ$ is strictly decreasing for known $n$ values, and decreases by about an order of magnitude per step of $n$. This suggests the growth of the number of sequencings of $\ZZ/n\ZZ$ as $n$ increases is sub-factorial. The best possible search algorithm for this space could be better than ours, but the time complexity of even the best search algorithm would be slower than exponential time but no slower than $O(n!)$.

The data above also leads us to believe that our approach of quitting our branch search once a partial sequence is found to be something that cannot give rise to a sequencing is a good practice. The ratio of sequencings to permutations is very low and decreasing, so eliminating partial permutations which are not going to result in sequencings as we go is a much better practice than generating all permutations and then searching that space.

This leads us to conclude that our algorithm has very little room for improvement, assuming that searching the space of permutations of $\ZZ/n\ZZ$ for sequencings is the only way to solve this problem of computing terms of the sequence. As we were able to derive new sequencings from known sequencings using Lemma \ref{lem:rel_prime_constructive}, and similar derivations have been discussed in Baker \cite{Baker}, one recommended direction for future progress for this sequence is attempting to identifying enough of these transformations to map a small set of known sequencings into the entire set of sequencings for $\ZZ/n\ZZ$, and to prove the completeness of these transformations.

\section*{Acknowledgements}
The authors would like to thank Enrique Trevi\~{n}o from Lake Forest College for suggesting a much more elegant wording of Lemma \ref{lem:branch} than our original statement. 

Concerned with sequence \seqnum{A141599}.

2020 \textit{Mathematics Subject Classifications}: 
Primary 05-08 Secondary 05A10, 05A05, 11B75
Keywords: Group sequencing, ordered group, combinatorics, permutation, recursive algorithm, modular arithmetic, additive group mod $n$ 


\begin{thebibliography}{99}
\bibitem{Gil}
E.~N. Gilbert,
``Latin squares which contain no repeated digrams,''
\emph{SIAM Review},
vol.~7,
pp.~189--198,
1965.

\bibitem{Gor}
B.~Gordon,
``Sequences in groups with distinct partial products,''
\emph{Pacific Journal of Mathematics},
vol.~11,
pp.~1309--1313,
1961.

\bibitem{OEIS}
OEIS Foundation Inc.,
``The On-Line Encyclopedia of Integer Sequences,''
2026.
Available at: \url{https://oeis.org}

\bibitem{Baker}
Z.~Baker,
``Properties and calculations of constructive orderings of $\mathbb{Z}/n\mathbb{Z}$,''
\emph{Minnesota Journal of Undergraduate Mathematics},
vol.~4,
2021.
Available at:
\url{https://pubs.lib.umn.edu/index.php/mjum/article/view/4154}

\end{thebibliography}
\end{document}